\documentclass [12pt,oneside]{amsart}
\usepackage{amsmath}

\usepackage[colorinlistoftodos]{todonotes}
\usepackage{verbatim}
\usepackage {amsfonts}
\usepackage {amsmath}
\usepackage {amsthm}
\usepackage {amssymb}
\usepackage {framed}
\usepackage {amsxtra}
\usepackage {enumerate}
\usepackage {graphicx}
\usepackage{graphicx}
\usepackage[framemethod=TikZ]{mdframed}
\usepackage{lipsum}
\mdfdefinestyle{MyFrame}{%
    linecolor=blue,
    outerlinewidth=2pt,
    roundcorner=20pt,
    innertopmargin=\baselineskip,
    innerbottommargin=\baselineskip,
    innerrightmargin=20pt,
    innerleftmargin=20pt,
    backgroundcolor=gray!50!white}
\usepackage{tcolorbox}
\usepackage[left=0.9in, right=0.9in, top=1 in, footskip=0.5 in, bottom=0.9 in]{geometry}
\makeatletter

\theoremstyle{definition}
\newtheorem{df}{Definition} [section]

\newtheorem{note}[df]{Note}

\newtheorem{remark}[df]{Remark}
\theoremstyle{plain}
\newtheorem{thm}[df]{Theorem}

\newtheorem{lemma}[df]{Lemma}

\newtheorem{fact}[df]{Fact}

\newtheorem{obs}[df]{Observation}

\usepackage{pgf,tikz,pgfplots}
\pgfplotsset{compat=1.14}
\usepackage{mathrsfs}
\usetikzlibrary{arrows}

\title{An Improved Upper Bound for the size of the Sphere of Influence Graph}

\author{Dan Ismailescu}
\address{Department of Mathematics, Hofstra University,
Hempstead, NY 11549, USA.}
\email{dan.p.ismailescu@hofstra.edu}

\author{Sung Hoon Kim}
\address{Bergen Catholic High School,1040 Oradell Avenue,
Oradell, NJ 07649, USA.}
\email{kelvin2kim@gmail.com}

\author{Taeyang David Park}
\address{Peddie School, 201 South Main Street,
Hightstown, NJ 08520, USA.}
\email{taeyangpark0801@gmail.com}
\begin{document}


\thispagestyle{empty}
\begin{abstract}

Let $V$ be a set of $n$ points in the plane. For each $x\in V$, let $B_x$ be the closed circular disk centered at $x$ with radius equal to the distance from $x$ to its closest neighbor. The {\it closed sphere of influence graph} on $V$ is defined as the undirected graph where $x$ and $y$ are adjacent if and only if the $B_x$ and $B_y$ have nonempty intersection.

It is known that every $n$-vertex closed sphere of influence graph has at most $cn$ edges, for some absolute positive constant $c$.
The first result was obtained in 1985 by Avis and Horton who provided the value $c=29$. Their result was successively improved by several authors: Bateman and Erd\H{o}s (c=18), Michael and Quint (c=17.5), and Soss (c=15).

In this paper we prove that one can take $c=14.5$.
\end{abstract}

\maketitle

\thispagestyle{empty}
\pagenumbering{arabic}

\section{\bf Introduction}
Let $V$ be a set of $n$ points in the plane. For each $x\in V$, let $B_x$ be the closed ball centered at $x$ with radius equal to the distance from $x$ to the nearest neighbor. We refer to these balls as the {\it spheres of influence} of the set $V$. The {\it open sphere of influence graph} on $V$ is defined as the undirected graph where $\{x,y\}$ is an edge if and only if the {\it interiors} of $B_x$ and $B_y$ have nonempty intersection. The {\it closed sphere of influence graph} is defined similarly; this time for $\{x,y\}$ to be an edge requires that $B_x$ and $B_y$ have nonempty intersection.

\begin{figure}\label{csig}
\definecolor{ffqqqq}{rgb}{1.,0.,0.}
\definecolor{ududff}{rgb}{0.30196078431372547,0.30196078431372547,1.}
\begin{tikzpicture}[line cap=round,line join=round,>=triangle 45,x=1.0cm,y=1.0cm,scale = 0.4]
\clip(-2,-10) rectangle (24.5,8.5);
\draw [line width=1.pt] (5.130791152195217,2.9236128828727286) circle (5.298346130874559cm);
\draw [line width=1.pt] (2.1208838183275462,-2.123302444622555) circle (5.876299556811006cm);
\draw [line width=1.pt] (7.806264337855369,-6.258124640642787) circle (6.2882263865225925cm);
\draw [line width=1.pt] (11.059396506783054,-0.8767751649399849) circle (4.4889645668944445cm);
\draw [line width=1.pt] (10.390528210368016,3.5620780749052643) circle (4.4889645668944445cm);
\draw [line width=1.pt] (17.413645322725916,3.9573184318777868) circle (5.464614434189294cm);
\draw [line width=1.pt] (16.92719565260589,-5.893287388052767) circle (6.081380906610993cm);
\draw [line width=1.pt] (20.180327821533574,-0.755162747409978) circle (5.464614434189294cm);
\draw [line width=2.pt,color=ffqqqq] (2.1208838183275462,-2.123302444622555)-- (5.130791152195217,2.9236128828727286);
\draw [line width=2.pt,color=ffqqqq] (2.1208838183275462,-2.123302444622555)-- (10.390528210368016,3.5620780749052643);
\draw [line width=2.pt,color=ffqqqq] (2.1208838183275462,-2.123302444622555)-- (11.059396506783054,-0.8767751649399849);
\draw [line width=2.pt,color=ffqqqq] (2.1208838183275462,-2.123302444622555)-- (7.806264337855369,-6.258124640642787);
\draw [line width=2.pt,color=ffqqqq] (5.130791152195217,2.9236128828727286)-- (11.059396506783054,-0.8767751649399849);
\draw [line width=2.pt,color=ffqqqq] (5.130791152195217,2.9236128828727286)-- (10.390528210368016,3.5620780749052643);
\draw [line width=2.pt,color=ffqqqq] (5.130791152195217,2.9236128828727286)-- (7.806264337855369,-6.258124640642787);
\draw [line width=2.pt,color=ffqqqq] (10.390528210368016,3.5620780749052643)-- (11.059396506783054,-0.8767751649399849);
\draw [line width=2.pt,color=ffqqqq] (10.390528210368016,3.5620780749052643)-- (17.413645322725916,3.9573184318777868);
\draw [line width=2.pt,color=ffqqqq] (17.413645322725916,3.9573184318777868)-- (11.059396506783054,-0.8767751649399849);
\draw [line width=2.pt,color=ffqqqq] (17.413645322725916,3.9573184318777868)-- (20.180327821533574,-0.755162747409978);
\draw [line width=2.pt,color=ffqqqq] (17.413645322725916,3.9573184318777868)-- (16.92719565260589,-5.893287388052767);
\draw [line width=2.pt,color=ffqqqq] (20.180327821533574,-0.755162747409978)-- (16.92719565260589,-5.893287388052767);
\draw [line width=2.pt,color=ffqqqq] (16.92719565260589,-5.893287388052767)-- (7.806264337855369,-6.258124640642787);
\draw [line width=2.pt,color=ffqqqq] (16.92719565260589,-5.893287388052767)-- (11.059396506783054,-0.8767751649399849);
\draw [line width=2.pt,color=ffqqqq] (7.806264337855369,-6.258124640642787)-- (11.059396506783054,-0.8767751649399849);
\draw [line width=2.pt,color=ffqqqq] (7.806264337855369,-6.258124640642787)-- (10.390528210368016,3.5620780749052643);
\draw [line width=2.pt,color=ffqqqq] (11.059396506783054,-0.8767751649399849)-- (20.180327821533574,-0.755162747409978);
\begin{scriptsize}
\draw [fill=ududff] (5.130791152195217,2.9236128828727286) circle (2.5pt);
\draw [fill=ududff] (11.059396506783054,-0.8767751649399849) circle (2.5pt);
\draw [fill=ududff] (10.390528210368016,3.5620780749052643) circle (2.5pt);
\draw [fill=ududff] (17.413645322725916,3.9573184318777868) circle (2.5pt);
\draw [fill=ududff] (20.180327821533574,-0.755162747409978) circle (2.5pt);
\draw [fill=ududff] (16.92719565260589,-5.893287388052767) circle (2.5pt);
\draw [fill=ududff] (7.806264337855369,-6.258124640642787) circle (2.5pt);
\draw [fill=ududff] (2.1208838183275462,-2.123302444622555) circle (2.5pt);
\end{scriptsize}
\end{tikzpicture}
\caption{A sphere of influence graph}
\end{figure}
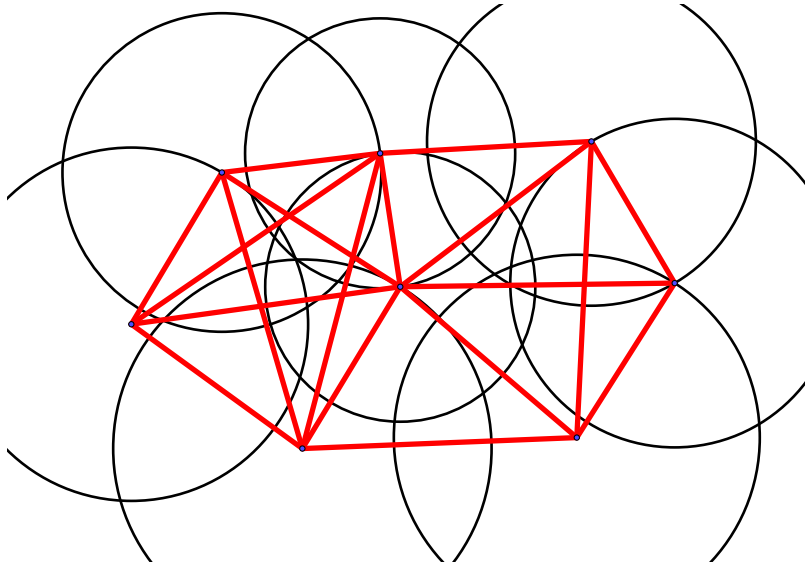

We will call these graphs OSIG or CSIG, respectively, depending on whether they are open or closed. If the distinction is irrelevant we will just call them SIG.
Sphere of influence graphs were introduced by Toussaint \cite{toussaint80} in 1980, as a type of proximity graph to model situations in pattern recognition and computer vision. The main question is to find a characterization of those graphs which can be realized as a SIG. This is a very difficult problem because an {\it induced subgraph} of a SIG need not be a SIG. In other words, being a SIG is not a hereditary property. For example, Jacobson, Lipman and McMorris proved \cite{JLM} that a tree is a CSIG if and only if it has a perfect matching. In particular, trees with an odd number of vertices cannot be CSIG. It is also easy to prove that the claw $K_{1,3}$ is not an OSIG, although the graph obtained by appending an edge to a leaf is an OSIG, see \cite{HJLM}.

Falling short of finding such a characterization, we may ask some easier questions. The two main open problems are
\begin{itemize}
\item{{\bf Problem 1.} What is the maximum number of edges of a SIG on $n$ vertices?}
\item{{\bf Problem 2.} What complete graphs can be realized as a SIG?}
\end{itemize}
Note that a SIG need not be a planar graph. For instance, the SIG shown in figure \ref{csig} contains a complete graph $K_5$. Toussaint \cite {toussaint80} asked whether any SIG has at most a linear number of edges. This was confirmed five years later by Avis and Horton \cite{AH}, who showed that the vertex corresponding to the smallest ball in any SIG has degree at most $29$. By induction, it follows that no SIG can contain more than $29n$ edges. In particular, no complete graph $K_{n}$ with $n\ge 31$ can be a SIG.

It was later realized that a stronger result had been proven (in a different form) more than thirty years earlier by Reifenberg \cite{R}, and independently by Bateman and Erd\H{o}s \cite {BE}. They showed that the vertex corresponding to the smallest ball in any SIG has degree at most $18$. Again, by induction it can be shown that no sphere of influence graph on $n$ vertices contains more than $18n$ edges. This is true for both OSIG and CSIG. Since a CSIG has at least as many edges as the corresponding OSIG, any upper bound for the size of the former is also going to be an upper bound for the size of the latter. From now on we restrict ourselves to studying CSIGs.

Michael and Quint \cite{MQ} further reduced this bound to $17.5n$ with the following beautiful argument. Let $x$ be the vertex with the smallest sphere of influence, of radius $r$. This sphere has radius $r$ because the nearest neighbor of $x$, say $y$, is distance $r$ away. Since the nearest neighbor relation is reflexive, $x$ is the nearest neighbor of $y$. Thus the spheres of influence of both $x$ and $y$ have radius $r$, the smallest radius over all spheres, so $x$ and $y$ each have at most $18$ neighbors. One edge is $\{x,y\}$ itself, and we have at most $35$ other edges. Performing induction on two vertices at a time instead of one yields a bound of $35n/2$ edges, or $17.5n$.

The best bound known today was established by Soss \cite{soss} in his 1998 Masters thesis.
\begin{thm}\cite{soss}\label{sossresult}
Every closed sphere of influence graph of order $n$ contains at most $15n$ edges.
\end{thm}
Soss' crucial idea was to introduce a {\it weighted} graph associated with the CSIG. We will present his approach later on.
What about the lower bound? It is easy to see that the hexagonal lattice has $18$ neighbors per vertex, for $9n$ edges in total.
Avis and Horton \cite {AH} conjectured that this is the largest number of edges possible for a CSIG.

Clearly, every single answer to the first problem translates into an answer for the second. For instance, Soss' result implies that
no complete graph $K_n$ with $n\ge 17$ is a CSIG. K\'{e}zdy and Kubicki \cite{KK} proved that $K_{12}$ is not a CSIG. From the other side,
it is known that every complete graph $K_n$ with $n\le 8$ can be realized as an OSIG. Harary et al. \cite{HJLM} conjectured that the smallest complete graph that is not a CSIG is $K_9$.

The problem was generalized to Euclidean spaces of arbitrary dimension by Guibas, Pach and Sharir \cite{GPS}. Among other things, they proved that the sphere of influence graph on $n$ vertices in $\mathbf{R}^d$ contains at most $5^d\cdot n$ edges.
In this paper we will improve Soss' 1998 result stated in Theorem \ref{sossresult}. We will prove the following
\begin{thm}\label{main}
Every closed sphere of influence graph of order $n$ contains at most $14.5n$ edges.
\end{thm}

\section{{\bf Reducing the problem to a bounded one}}
We will prove the following stronger result.
\begin{thm}\label{weight}
Let $p\ge 1$ be a constant to be specified later. Let $\mathcal{P}_1$ be a finite set of points in the annulus $p\le \rho \le 1+p$, and let $\mathcal{P}_{1/2}$ be a finite set of points in the annulus $1\le \rho \le 1+p$ with the following properties:
\begin{itemize}
\item{$\mathcal{P}_1$ and $\mathcal{P}_{1/2}$ are disjoint.}
\item{every point in $\mathcal{P}_1$ has weight $1$, and every point in $\mathcal{P}_{1/2}$ has weight $1/2$.}
\item{the distance between any two points of weight $1$ is at least $p$; the distance between two points of weight $1/2$ is at least $q=1/p$; the distance between any two points of different weights is at least $p$.}
\end{itemize}
Then the total weight of all the points in $\mathcal{P}_1 \cup \mathcal{P}_{1/2}$ cannot exceed $14.5$.
\end{thm}

In this section we combine the reasoning of Bateman and Erd\H{o}s \cite{BE} with that of Soss \cite{soss} to prove that Theorem \ref{weight} implies Theorem \ref{main}. Let $p\ge 1$ be a constant to be chosen later, and denote $q=1/p$. We start by assigning weights to the edges of the CSIG as follows. First, we replace each undirected edge $\{a,b\}$ with two directed edges
$(a,b)$ and $(b,a)$. Let the radii of the spheres of influence of $a$ and $b$ be $r_a$ and $r_b$. Then define $w(a,b)$, the weight of $(a,b)$, as

\[  w(a,b)= \left\{
\begin{array}{ll}
      1 &  \quad \text{if}\quad p< r_b/r_a \\
      1/2 &  \quad \text{if}\quad q\le r_b/r_a\le p\\
      0 &  \quad \text{if}\quad r_b/r_a<q\\
\end{array}
\right. \]
We refer to this graph as the {\it weighted sphere of influence graph}, or WSIG.
\begin{obs}
For any point set $V$, the total weight of all the directed edges in the WSIG of $V$ is equal to the number of edges in the CSIG of $V$.
\end{obs}
This is immediate since for any two adjacent vertices $a$ and $b$ in the CSIG of $V$ we have that $w(a,b)+w(b,a)=1$.
The above observation implies that if we can prove that no WSIG on $n$ vertices has edges whose total weight is greater than $14.5n$, then we have also proven Theorem \ref{main}. This is exactly the method behind our proof, which we state below for future reference.
\begin{thm}\label{point}
There exists no node in the WSIG for which the weights of the outgoing edges sum to more than $14.5$.
\end{thm}

We will next show how Theorem \ref{point} can be obtained from Theorem \ref{weight}. We require the following lemma,  whose first part appears in \cite{BE}. The second part is very similar, and this is most likely the reason Bateman and Erd\H{o}s left it out.
\begin{lemma}{\bf The radial projection lemma.}\label{radprojlemma}
In polar coordinates, let $X=(x,\theta_x)$ and $Y=(y,\theta_y)$ be the centers of two circles of radii $r_x$ and $r_y$, such that they do not contain each other's centers but they both intersect the unit circle $\rho=1$. Let $R>1$ be a fixed number.

(a) If $X$ and $Y$ lie outside the disc $\rho\le R$, then the points $A=(R,\theta_x)$ and $B=(R,\theta_y)$ are at least distance $R-1$ apart.

(b) If $X$ lies inside the annulus $1\le \rho\le R$ and $Y$ lies outside the disc $\rho\le R$, then points $X=(x,\theta_x)$ and $B=(R,\theta_y)$ are at least distance $R-1$ apart.
\end{lemma}
\begin{proof}
We first prove part (a). In this case we have $x\ge R$ and $y\ge R$. Since the two circles do not contain each other's centers, it follows that
$XY\ge \max\{r_x,r_y\}$, see figure \ref{f1}.
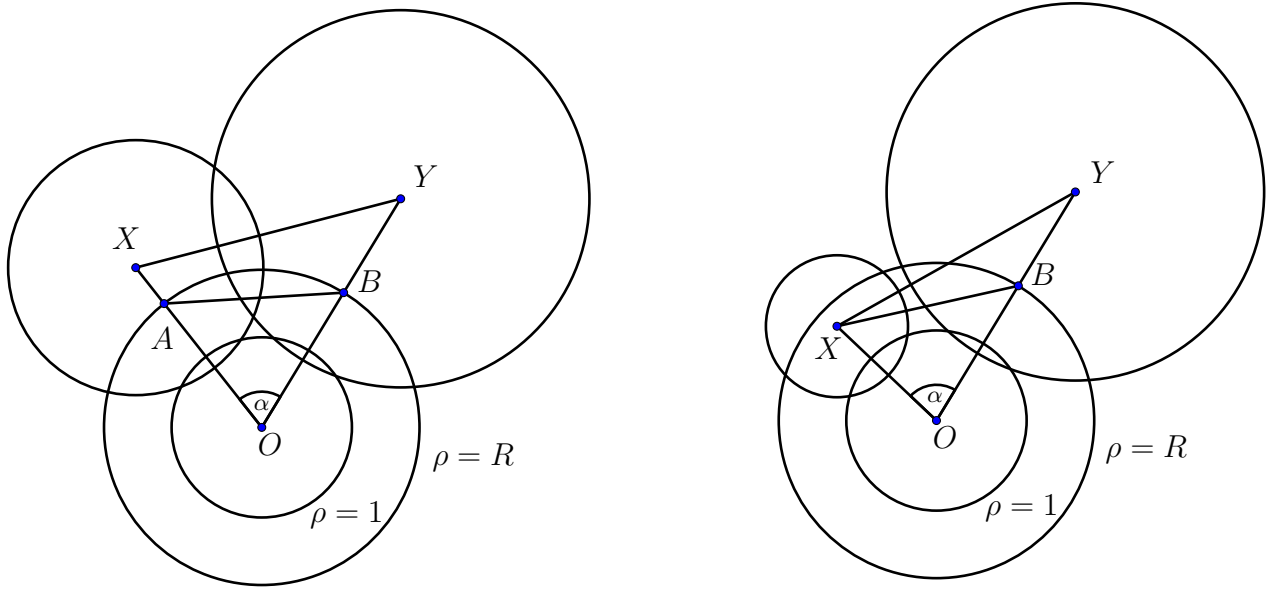
\begin{figure}\label{f1}
\definecolor{qqqqff}{rgb}{0.,0.,1.}
\begin{tikzpicture}[line cap=round,line join=round,>=triangle 45,x=1.0cm,y=1.0cm,scale = 0.5]
\clip(-7,-8.69424522959205) rectangle (29.27935485055309,12.845864818963152);
\draw [shift={(0.2518187327866559,-2.256638930350115)},line width=1.pt] (0,0) -- (58.737378324275156:0.9512527689559348) arc (58.737378324275156:128.2449333975739:0.9512527689559348) -- cycle;
\draw [shift={(18.191198401867105,-2.074376670667764)},line width=1.pt] (0,0) -- (58.73737832427517:0.9512527689559348) arc (58.73737832427517:136.522467731053:0.9512527689559348) -- cycle;
\draw [line width=1.pt] (0.2518187327866559,-2.256638930350115) circle (2.3972936080448304cm);
\draw [line width=1.pt] (3.94482196353805,3.826220272052222)-- (0.2518187327866559,-2.256638930350115);
\draw [line width=1.pt] (0.2518187327866559,-2.256638930350115) circle (4.1940639218668485cm);
\draw [line width=1.pt] (-3.0987169567937234,1.9942684452394692)-- (0.2518187327866559,-2.256638930350115);
\draw [line width=1.pt] (-3.0987169567937234,1.9942684452394692) circle (3.3927458112488926cm);
\draw [line width=1.pt] (-3.0987169567937234,1.9942684452394692)-- (3.94482196353805,3.826220272052222);
\draw [line width=1.pt] (-2.3444095985645923,1.0372600711668178)-- (2.428376746864058,1.3284367162800323);
\draw [line width=1.pt] (3.94482196353805,3.826220272052222) circle (5.0189948135229585cm);
\draw (-4.014492062904631,3.3606028093637326) node[anchor=north west] {$X$};
\draw (4,5) node[anchor=north west] {$Y$};
\draw (2.5,2.2) node[anchor=north west] {$B$};
\draw (-3,0.7) node[anchor=north west] {$A$};
\draw (-0.1460641358171623,-2.127525376792453) node[anchor=north west] {$O$};
\draw (1.25,-3.9357525826358204) node[anchor=north west] {$\rho = 1$};
\draw (4.5,-2.381311651296785) node[anchor=north west] {$\rho = R$};
\draw [line width=1.pt] (18.191198401867105,-2.074376670667764) circle (2.3972936080448304cm);
\draw [line width=1.pt] (21.8842016326185,4.0084825317345745)-- (18.191198401867105,-2.074376670667764);
\draw [line width=1.pt] (18.191198401867105,-2.074376670667764) circle (4.1940639218668485cm);
\draw [line width=1.pt] (15.545125087675975,0.43468184650734826)-- (18.191198401867105,-2.074376670667764);
\draw [line width=1.pt] (15.545125087675975,0.43468184650734826)-- (21.8842016326185,4.0084825317345745);
\draw [line width=1.pt] (21.8842016326185,4.0084825317345745) circle (5.0189948135229585cm);
\draw (14.661770634263558,0.4420606525639113) node[anchor=north west] {$X$};
\draw (22.,5.1) node[anchor=north west] {$Y$};
\draw (20.4,2.4) node[anchor=north west] {$B$};
\draw (17.800904771818143,-1.9371856709142037) node[anchor=north west] {$O$};
\draw (19.2,-3.7454128767575714) node[anchor=north west] {$\rho = 1$};
\draw (22.4,-2.190971945418536) node[anchor=north west] {$\rho = R$};
\draw [line width=1.pt] (15.545125087675975,0.43468184650734826) circle (1.889114515272353cm);
\draw [line width=1.pt] (15.545125087675975,0.43468184650734826)-- (20.367756415944505,1.5106989759623828);
\begin{scriptsize}
\draw [fill=qqqqff] (0.2518187327866559,-2.256638930350115) circle (3.0pt);
\draw [fill=qqqqff] (3.94482196353805,3.826220272052222) circle (3.0pt);
\draw [fill=qqqqff] (-3.0987169567937234,1.9942684452394692) circle (3.0pt);
\draw [fill=qqqqff] (-2.3444095985645923,1.0372600711668178) circle (3.0pt);
\draw [fill=qqqqff] (2.428376746864058,1.3284367162800323) circle (3.0pt);
\draw [fill=qqqqff] (18.191198401867105,-2.074376670667764) circle (3.0pt);
\draw [fill=qqqqff] (21.8842016326185,4.0084825317345745) circle (3.0pt);
\draw [fill=qqqqff] (15.545125087675975,0.43468184650734826) circle (3.0pt);
\draw [fill=qqqqff] (20.367756415944505,1.5106989759623828) circle (3.0pt);
\draw[color=black] (0.25,-1.65) node {$\alpha$};
\draw[color=black] (18.15,-1.45) node {$\alpha$};
\end{scriptsize}
\end{tikzpicture}
\vspace{-1.3cm}
\caption{Illustration for Lemma 2.4}
\end{figure}
On the other hand, both circles intersect $\rho=1$, which means that $x\le 1+r_x$ and $y\le 1+r_y$. It follows that
$XY\ge \max\{x-1,y-1\}$. Without loss of generality one can assume $1<R\le x\le y$. By the law of cosines in triangle $OXY$
\begin{equation*}
\cos{\alpha}=\frac{x^2+y^2-XY^2}{2xy}\le \frac{x^2+y^2-(y-1)^2}{2xy}=\frac{1}{x}+\frac{x^2-1}{2xy}\le\frac{1}{x}+\frac{x^2-1}{2x^2}.
\end{equation*}
Let $f(x)=\frac{1}{x}+\frac{x^2-1}{2x^2}$. Differentiating, we obtain that $f'(x)=(1-x)/x^3<0$ for all $x\ge R>1$. Hence, the maximum of $f$ is reached when $x=R$. We obtain
\begin{equation*}
\cos{\alpha}\le \frac{1}{R}+\frac{R^2-1}{2R^2} \longrightarrow 1-\cos{\alpha}\ge \frac{(R-1)^2}{2R^2}.
\end{equation*}
On the other hand,
\begin{equation*}
AB^2=2R^2(1-\cos{\alpha})\ge 2R^2\cdot\frac{(R-1)^2}{2R^2}=(R-1)^2\longrightarrow AB\ge R-1,\,\, \text{as desired}.
\end{equation*}

This proves part (a). For part (b) we are given that $1\le x\le R\le y$. As before, $XY\ge \max\{r_x,r_y\}\ge \max\{x-1,y-1\} =y-1$.
It follows that
\begin{equation*}
\cos{\alpha}=\frac{x^2+y^2-XY^2}{2xy}\le \frac{x^2+y^2-(y-1)^2}{2xy}=\frac{1}{x}+\frac{x^2-1}{2xy}\le\frac{1}{x}+\frac{x^2-1}{2xR}.
\end{equation*}
From the law of cosines in triangle $OAY$ we obtain that

\begin{equation*}
XB^2=x^2+R^2-2xR\,\cos\alpha\ge x^2+R^2-2xR\cdot\left(\frac{1}{x}+\frac{x^2-1}{2xR}\right)=(R-1)^2,
\end{equation*}
from which $XB\ge R-1$ as claimed. This completes the proof of Lemma \ref{radprojlemma}.
\end{proof}
We are going to use this lemma on a couple of occasions. Note that only the points that are outside the circle $\rho=R$ are radially projected onto the circle. The points that are inside the circle stay fixed - see case (b). We are now ready to prove the following result.
\begin{thm}{\bf modeled after a result of Soss\cite{soss}}
Theorem \ref{weight} implies Theorem \ref{point}.
\end{thm}
\begin{proof}
Let $O$ be some node in the WSIG. Without loss of generality, assume that $O$ is at the origin and that the sphere of influence of $O$ is $1$.
Consider the outgoing edges from $O$ that have nonzero weight. We have a set of circles of radius at least $q$, which all intersect the circle $\rho=1$ such that the center of no circle is contained in any other. Since $r_O=1$, no circle is centered in $\rho<1$. Also, one cannot have two centers that have the same amplitude.
Indeed, assume for the sake of contradiction that $X(x,\theta)$ and $Y(y,\theta)$ are two such points with say, $x<y$. Since the circle centered at $Y$ has to intersect $\rho=1$, it is going to contain $X$ in its interior, impossible.

Let us look first at the outgoing edges of weight $1$ from $O$. These edges correspond to circles of radius at least $p$. Since no such circle contains $O$, these centers are all located in $\rho\ge p$. For obvious reasons, we will refer to these centers as {\it points of weight $1$}. Clearly, since no circle may contain the center of another, the distance between any two points of weight $1$ is at least $p$.

Similarly, consider the outgoing edges of weight $1/2$ from $O$. These edges correspond to circles of radii in the interval $[q,p]$. All these centers are located in the annulus $1\le \rho\le 1+p$. The left bound is due to the fact that none of these centers may be contained in the open disc $\rho<1$, while the right bound follows from the fact that each circle must intersect $\rho=1$.  We will refer to these circle centers as {\it points of weight $1/2$}. Clearly, since no circle may contain the center of another, the distance between any two points of weight $1/2$ is at least $q$.

In addition, the distance between a point of weight $1$ and one of weight $1/2$ must be at least $p$, since no circle centered at a point of weight $1$ can contain a point of weight $1/2$ in its interior.

 In summary, we have two kinds of points: points of weight $1$ all lying in $\rho\ge p$, and points of weight $1/2$, all located in the annulus
$1 \le \rho\le 1+p$. The distance between any two points of weight $1$ is at least $p$, the distance between two points of weight $1/2$ is at least $q$, and the distance between any two points of different weights is at least $p$.

 While the points of weight $1/2$ are restricted to the annulus $1\le \rho \le 1+p$, the points of weight $1$ can be arbitrarily far away from the origin. We now use Lemma \ref{radprojlemma} for the first time. Divide the set of points of weight $1$ in two classes: {\it interior points} are those which are within $p\le \rho \le 1+p$, and {\it exterior points} are those outside the circle $\rho=1+p$.
Next, radially project the exterior points onto the circle $\rho=1+p$. In other words, if $X(x,\theta_x)$ is such a point with $x>1+p$, its radial projection is $A(1+p,\theta_x)$. The interior points of weight $1$ and the points of weight $1/2$ stay fixed.

We call these newly obtained points {\it boundary points} of weight $1$. From Lemma \ref{radprojlemma} part (a) with $R=1+p$, it follows that the distance between two boundary points of weight $1$ is at least $1+p-1=p$. Also, from Lemma \ref{radprojlemma} part (b) it follows that the distance between a boundary point of weight $1$, and any other non-boundary point (interior point of weight $1$, or point of weight $1/2$) is at least $p$, as well.

We have now arrived at the hypotheses from Theorem \ref{weight}. Here $\mathcal{P}_1$ consists of all the interior and boundary points of weight $1$, while $\mathcal{P}_{1/2}$ consists of all points of weight $1/2$. Thus if we can prove Theorem \ref{weight}, then Theorem \ref{point}, and implicitly Theorem \ref{main}, will follow.
\end{proof}

\begin{note}
It now remains to prove Theorem \ref{weight}. Bateman and Erd\H{o}s \cite{BE} used $p=q=1$ to prove an upper bound of $18$ for the total weight
of the points in $\mathcal{P}_1 \cup \mathcal{P}_{1/2}$. For this selection of $p$ and $q$, this bound is optimal. Soss \cite{soss} used $p=3/2$, $q=2/3$ to prove an upper bound of of $15$ for the total weight of the points in $\mathcal{P}_1 \cup \mathcal{P}_{1/2}$. In a private communication, Soss admitted he chose these values because they were ``nice." Our idea was to try to choose $p$ so that we can optimize, or at least improve Soss' result. For reasons that are going to become clear later, we ended up choosing $p=1.409$, $q=1/p=0.7097\ldots$. We record this choice for future reference. For the remainder of the paper
\begin{equation}\label{pq}
p=1.409\qquad q=1/p=0.7097\ldots
\end{equation}
\end{note}

The proof of Theorem \ref{weight} is based on a technique initiated by Bateman and Erd\H{o}s \cite{BE} and extended by Soss \cite{soss}, of fitting points into annuli. We omit the simple proof.
\begin{lemma}\label{BEsoss}\cite{BE,soss}
Label the origin as $O$. Let $r$, $R$, and $d$ be such that $0\le R-d\le r\le R$. Suppose we have two points $A$ and $B$ which lie in the annulus $r\le \rho\le R$ and which have mutual distance at least $d$. Then the minimum value of the angle $\angle AOB$ is at least
\begin{equation}\label{phi}
\Phi_d(r,R)=\min\left(\arccos\frac{R^2+r^2-d^2}{2Rr},2\arcsin\frac{d}{2R}\right)
\end{equation}
\end{lemma}
The idea for proving Theorem \ref{weight} is simple. We first sort the points with respect to their amplitudes - we call this a {\it circular ordering}. Note that such an ordering is always well defined as no two points can have the same amplitude. Next, we compute the minimum angles between each pair of consecutive points.

For example the minimum angle between a point of weight $1$ in $1.7\le \rho \le 2$, and another of weight $1$ in $1.5\le \rho\le 1.8$ is at least $\Phi_p(1.5,2)$. This is because both points belong to the annulus $1.5\le \rho\le 2$, and they are at least distance $p$ apart.

Likewise, the minimum angle between a point of weight $1$ in $1.7\le \rho \le 2.2$ and another of weight $1/2$ in $1.3\le \rho\le 1.8$ is at least $\Phi_p(1.3,2.2)$. This is because both points belong to the annulus $1.3\le \rho\le 2.2$, and they are at least distance $p$ apart.

In general, the following is true:
\begin{remark}\label{one}
If one has two points, one in $a\le \rho \le b$, the other in $c\le \rho \le d$, and {\bf at least one of them is of weight $1$}, then the minimum angle is at least $\Phi_p\left(\min(a,c),\max(b,d)\right)$.
\end{remark}

If both points are of weight $1/2$ we do things a bit differently. We believe that the following result is very important since for quite a while we were confused about how this case works . Recall first that the distance between any two points of weight $1/2$ is at least $q=0.709\ldots$, and all these points are in the annulus $1\le \rho \le 1+p$. We claim that the following holds.

\begin{lemma}\label{two}
If one has {\bf two points of weight $1/2$} , one in $a\le \rho \le b$, the other in $c\le \rho \le d$, and $\min(a,c)\le 1+q$, then the minimum angle is at least $\Phi_q\left(\min(a,c), \min(1+q,\max(b,d)\right)$.
\end{lemma}
\begin{proof}
Notice the subtle differences between the result in remark \ref{one} and that of lemma \ref{two}. Let $X(x,\theta_x)$ and $Y(y,\theta_y)$ located in $a\le \rho\le b$ and $c\le \rho \le d$, respectively, be the two points of weight $1/2$. Recall that $XY\ge q$. If both $x\le 1+q$ and $y\le 1+q$, then both points belong to the annulus $\min(a,c)\le \rho \le \min(1+q,\max(b,d))$, and the result follows.

If both $x\ge 1+q$ and $y\ge 1+q$, then we have that $1+q\le \max(b,d)$. Next, we radially project $X$ and $Y$ onto the circle $\rho=1+q$. We thus obtain the points $A(1+q,\theta_x)$ and $B(1+q,\theta_y)$. By Lemma \ref{radprojlemma} part (a) for $R=1+q$, we have that $AB\ge 1+q-1=q$. On the other hand, it is clear that $\angle XOY=\angle AOB$. But since $AB\ge q$ and $OA=OB=1+q$, it then follows that $\angle XOY\ge \Phi_q\left(\min(a,c), 1+q\right)=\Phi_q\left(\min(a,c), \min(1+q,\max(b,d))\right)$.

Finally, suppose that $x\le 1+q$ and $y\ge 1+q$, so again $1+q\le \max(b,d)$. Leave $X$ fixed and radially project $Y$ onto the circle $\rho=1+q$, to obtain point $B(1+q,\theta_y)$. Then by Lemma \ref{radprojlemma} part (b) for $R=1+q$, we have that $XB\ge 1+q-1=q$. On the other hand, it is clear that $\angle XOY=\angle XOB$. But since $XB\ge q$ and $OX=OB=1+q$, it then follows that $\angle XOY\ge \Phi_q\left(\min(a,c), 1+q\right)=\Phi_q\left(\min(a,c), \min(1+q,\max(b,d)\right)$.
\end{proof}

\begin{figure}\label{target}
\definecolor{ududff}{rgb}{0.30196078431372547,0.30196078431372547,1.}
\begin{tikzpicture}[line cap=round,line join=round,>=triangle 45,x=1.0cm,y=1.0cm,scale = 2]
\clip(-4.25,-3.0973014929209253) rectangle (4.521301579756383,3.55637144672084);
\draw [line width=2.pt] (0.,0.) circle (1.2cm);
\draw [line width=2.pt] (0.,0.) circle (1.5cm);
\draw [line width=2.pt] (0.,0.) circle (2.cm);
\draw [line width=2.pt] (0.,0.) circle (2.4cm);
\draw [line width=2.pt] (0.,0.)-- (0.8539927366752869,1.0433887275300697);
\draw [line width=2.pt] (0.,0.)-- (-1.328308069517886,1.1353630910094754);
\draw [line width=2.pt] (0.,0.)-- (2.225246883095518,-0.15227799770220463);
\draw [line width=2.pt] (0.,0.)-- (0.9961349347798231,-2.008487878832029);
\draw [line width=2.pt] (0.,0.)-- (-1.4119211272264367,-0.9549633517042909);
\draw (0.9,1.5) node[anchor=north west] {\large $D$};
\draw (2,-0.25) node[anchor=north west] {\large$C$};
\draw (1.05,-1.65) node[anchor=north west] {\large$B$};
\draw (0.21019396501807971,-0.08) node[anchor=north west] {\large$O$};
\draw (-1.67,1.1) node[anchor=north west] {\large $E$};
\draw (-1.75,-0.5381965161356314) node[anchor=north west] {\large$A$};
\draw (-0.2,1.15) node[anchor=north west] {\large$1.2$};
\draw (-0.2,1.8) node[anchor=north west] {\large$1.5$};
\draw (-0.2,2.29) node[anchor=north west] {\large$2.0$};
\draw (-0.2,2.7) node[anchor=north west] {\large$2.4$};
\begin{scriptsize}
\draw [fill=ududff] (0.,0.) circle (1.5pt);
\draw [fill=ududff] (0.8539927366752869,1.0433887275300697) circle (1.5pt);
\draw [fill=ududff] (2.225246883095518,-0.15227799770220463) circle (1.5pt);
\draw [fill=ududff] (0.9961349347798231,-2.008487878832029) circle (1.5pt);
\draw [fill=ududff] (-1.4119211272264367,-0.9549633517042909) circle (1.5pt);
\draw [fill=ududff] (-1.328308069517886,1.1353630910094754) circle (1.5pt);
\end{scriptsize}
\end{tikzpicture}
\vspace{-1.75cm}
\caption{A sample configuration of weighted points}
\end{figure}
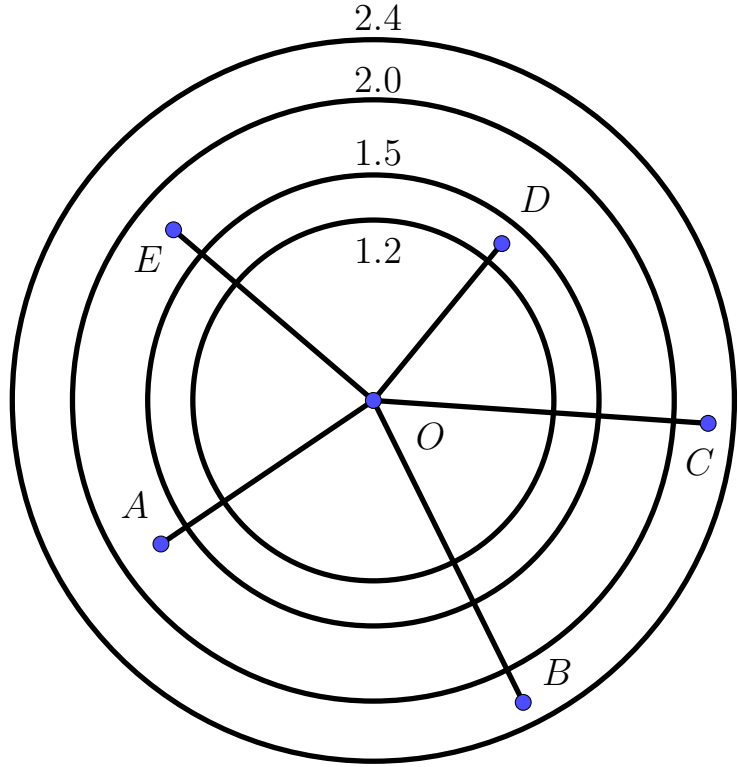
A sample configuration is presented in figure \ref{target}. The three white circles have weight $1$, the other two weight $1/2$. We see that
$\angle AOB\ge \Phi_p(1.5,2,4), \angle BOC \ge \Phi_p(2,2.4), \angle COD\ge \Phi_p(1.2,2.4), \angle DOE \ge \Phi_q(1.2,1+q)$, and $\angle EOA\ge \Phi_p(1.5,2)$.
\section{\bf Five easy cases}
We now have everything in place to begin proving Theorem \ref{weight}. The proof of each result below proceeds by assuming that a certain configuration of weight $15$ is possible and then proving that the sum of the angles subtended by pairs of consecutive points in the circular ordering is greater than $360^{\circ}$, a contradiction. Relations (\ref{pq}), (\ref{phi}), remark \ref{one} and lemma \ref{two} are going to be used extensively in the next sections, many times without explicit reference.  The proof is quite involved as it requires a rather lengthy case analysis. However, it should be rather straightforward to verify each individual statement.
\begin{lemma}\label{120}
It is impossible to have $12$ points of weight $1$ in the annulus $p\le \rho \le 1+p$.
\end{lemma}
\begin{proof}
Suppose there are $12$ points of weight $1$. The angle determined by two consecutive points is at least
$\Phi_p\left(p,1+p\right)>31^{\circ}.25$. The conclusion follows since $12\,\Phi_p\left(p,1+p\right)>375^{\circ}$.
\end{proof}
\begin{lemma}\label{027}
It is impossible to have $27$ points of weight $1/2$ in the annulus $1\le \rho \le 1+p$.
\end{lemma}
\begin{proof}
Consider first the points of weight $1/2$ located in the annulus $1\le \rho \le 1.25$. The angle between any
two consecutive points is at least $\Phi_q\left(1,1.25\right)=32^{\circ}.98\ldots$. Since $11\cdot\Phi_q\left(1,1.25\right)>362^{\circ}$, it follows that there are at most $10$ points of weight $1/2$ in $1\le \rho \le 1.25$.

Next let us consider the remaining points of weight $1/2$, those in $1.25\le \rho\le 1+p$. The angle between
any two consecutive such points is at least $\Phi_q\left(1.25, 1+q\right)=21^{\circ}.31\ldots$. Since $17\cdot \Phi_q\left(1.25, 1+q\right)>362^{\circ}$, there are at most $16$ points of weight $1/2$ in $1.25\le \rho\le 1+p$.
Hence, there cannot be more than $10+16=26$ points of weight $1/2$.
\end{proof}
\begin{lemma}\label{224}
It is impossible to have $2$ points of weight $1$ in $p\le \rho \le 1+p$ and $24$ points of weight $1/2$ in $1\le \rho \le 1+p$.
\end{lemma}
\begin{proof}
We will need the following
\begin{fact}\label{f224}
It is impossible to have $2$ points of weight $1$  in $p\le \rho \le 1+p$ and $13$ points of weight $1/2$ in $1.32\le\rho\le 1+p$.
\end{fact}
We prove the above fact first. Suppose we may place $15$ points as stated above. Sort the points in circular order and look at the $15$ angles determined by pairs of consecutive points.

Let $n_{11}$ be the number of angles determined by two consecutive points of weight $1$ (if any). Each of these angles is at least $\Phi_p\left(p,1+p\right)=31^{\circ}.25\ldots$.

Similarly, let $n_{hh}$ denote the number of angles determined by two consecutive points of weight $1/2$. Each of these angles is at least
$\Phi_q\left(1.32, 1+q\right)=22^{\circ}.77\ldots$.

Finally, let $n_{1h}$ be the number of angles determined by two consecutive points of different weight. Each of these angles is at least $\Phi_p\left(1.32,1+p\right)=29^{\circ}.03\ldots$.
Hence, the sum of the $15$ angles around the origin is at least $\mathcal{A}=n_{11}\cdot 31^{\circ}.25+n_{hh}\cdot 22^{\circ}.77+n_{1h}\cdot 29^{\circ}.03$.

We may have $[n_{11},n_{hh},n_{1h}]=[0,11,4]$ or $[n_{11},n_{hh},n_{1h}]=[1,12,2]$, depending on whether the two points of weight $1$ are consecutive in the circular ordering or not.

In the first case we have $\mathcal{A}>366^{\circ}$, in the second case $\mathcal{A}>362^{\circ}$. We reached the desired contradiction. This concludes the proof of fact \ref{f224}.

We continue with the proof of Lemma \ref{224}. Suppose the statement is false. Then, by fact \ref{f224} it follows that at least $12$ points of weight $1/2$ are in the annulus $1\le \rho \le 1.32$. However, the angle determined by two consecutive such points is at least
$\Phi_q\left(1,1.32\right)=31^{\circ}.18\ldots$. Since $12\cdot \Phi_q\left(1,1.32\right)>374^{\circ}$, the contradiction follows.
\end{proof}
\begin{lemma}\label{421}
It is impossible to have $4$ points of weight $1$ in $p\le \rho \le 1+p$ and $21$ points of weight $1/2$ in $1\le \rho \le 1+p$.
\end{lemma}
\begin{proof}
We will need the following
\begin{fact}\label{f421}
It is impossible to have $4$ points of weight $1$ in $p\le \rho \le 1+p$, and $11$ points of weight $1/2$ in $1.25\le \rho\le 1+p$.
\end{fact}
We use the same approach as in the proof of fact \ref{f224}. Sort the points in circular order, and look at the $15$ angles determined by pairs of consecutive points.

Let $n_{11}$ be the number of angles determined by two consecutive points of weight $1$. Each of these angles is at least $\Phi_p\left(p,1+p\right)=31^{\circ}.25\ldots$.

Similarly, let $n_{hh}$ denote the number of angles determined by two consecutive points of weight $1/2$. Each of these angles is at least
$\Phi_q\left(1.25, 1+q\right)=21^{\circ}.31\ldots$.

Finally, let $n_{1h}$ be the number of angles determined by two consecutive points of different weight. Each of these angles is at least $\Phi_p\left(1.25,1+p\right)=26^{\circ}.69\ldots$.
Hence, the sum of the $15$ angles around the origin is at least $
\mathcal{A}=n_{11}\cdot 31^{\circ}.25+n_{hh}\cdot 21^{\circ}.31+n_{1h}\cdot 26^{\circ}.69$.

We have four possibilities for $[n_{11},n_{hh},n_{1h}]$: $[0, 7, 8], [1, 8, 6], [2, 9, 4]$, or $[3, 10, 2]$. The minimum of $\mathcal{A}$ is reached when $[n_{11},n_{hh},n_{1h}]=[3, 10, 2]$, which corresponds to the case when the four points of weight $1$ are consecutive. But even in this case we have $\mathcal{A}=3\cdot 31^{\circ}.25 +10\cdot 21^{\circ}.31 +2\cdot 26^{\circ}.69=360^{\circ}.23>360^{\circ}$, a contradiction.
The proof of fact \ref{f421} is complete.

We return now to the proof of Lemma \ref{421}.
Suppose the statement is false. Then, by fact \ref{f421} it follows that at least $11$ points of weight $1/2$ are in the annulus $1\le \rho \le 1.25$. However, the angle determined by two consecutive such points is at least
$\Phi_q\left(1,1.25\right)=32^{\circ}.98\ldots$. Since $11\cdot \Phi_q\left(1,1.25\right)>362^{\circ}$, the contradiction follows.
\end{proof}
\begin{lemma}\label{520}
It is impossible to have $5$ points of weight $1$ in $p\le \rho \le 1+p$ and $20$ points of weight $1/2$ in $1\le \rho \le 1+p$.
\end{lemma}
\begin{proof}
As earlier, we start with a
\begin{fact}\label{f520}
It is impossible to have $5$ points of weight $1$ in $p\le \rho \le 1+p$ and $10$ points of weight $1/2$ in $1.23\le \rho\le 1+p$.
\end{fact}
We follow the same approach as in facts \ref{f224} and \ref{f421}.

Let $n_{11}$ be the number of angles determined by two consecutive points of weight $1$. Each of these angles is at least $\Phi_p\left(p,1+p\right)=31^{\circ}.25\ldots$.

Similarly, let $n_{hh}$ denote the number of angles determined by two consecutive points of weight $1/2$. Each of these angles is at least
$\Phi_q\left(1.23, 1+q\right)=20^{\circ}.77\ldots$.

Finally, let $n_{1h}$ be the number of angles determined by two consecutive points of different weight. Each of these angles is at least $\Phi_p\left(1.23,1+p\right)=25^{\circ}.90\ldots$.

Hence, the sum of the $15$ angles around the origin is at least $\mathcal{A}=n_{11}\cdot 31^{\circ}.25+n_{hh}\cdot 20^{\circ}.77+n_{1h}\cdot 25^{\circ}.9$.

We have five possibilities for $[n_{11},n_{hh},n_{1h}]$: $[0, 5, 10], [1, 6, 8], [2, 7, 6], [3, 8, 4]$, or $[4,9,2]$.

The minimum of $\mathcal{A}$ is reached when $[n_{11},n_{hh},n_{1h}]=[0, 5, 10]$, which corresponds to the case when no two points of weight $1$ are consecutive. But even in this case we have $\mathcal{A}=5\cdot 20^{\circ}.77 +10\cdot 25^{\circ}.9=362^{\circ}.85>360^{\circ}$, a contradiction.

The proof of fact \ref{f520} is complete. We can now prove Lemma \ref{520}.
Suppose the statement is false. Then, by fact \ref{f520} it follows that at least $11$ points of weight $1/2$ are in the annulus $1\le \rho \le 1.23$. However, the angle determined by two consecutive such points is at least
$\Phi_q\left(1,1.23\right)=33^{\circ}.53\ldots$. Since $11\cdot \Phi_q\left(1,1.23\right)>368^{\circ}$, the contradiction follows.
\end{proof}
Therefore the following cases suffice to prove Theorem \ref{weight}. We will show that one cannot have
\begin{itemize}
\item{$6$ points of weight $1$ and $18$ of weight $1/2$.}
\item{$7$ points of weight $1$ and $16$ of weight $1/2$.}
\item{$8$ points of weight $1$ and $14$ of weight $1/2$.}
\item{$9$ points of weight $1$ and $12$ of weight $1/2$.}
\item{$10$ points of weight $1$ and $10$ of weight $1/2$.}
\item{$11$ points of weight $1$ and $8$ of weight $1/2$.}
\end{itemize}
The proofs are contained in the following sections.


\section{\bf 6 points of weight 1 and 18 points of weight 1/2}
\begin{fact}\label{f618akelvin}
It is impossible to have $6$ points of weight $1$ in $p\le \rho \le 1+p$ and $8$ points of weight $1/2$ in $1.259\le \rho\le 1+p$.
\end{fact}
\begin{proof}
Let $n_{11}$ be the number of angles determined by two consecutive points of weight $1$. Each of these angles is at least $\Phi_p\left(p,1+p\right)=31^{\circ}.25\ldots$.

Let $n_{hh}$ denote the number of angles determined by two consecutive points of weight $1/2$. Each of these angles is at least
$\Phi_q\left(1.259, 1+q\right)=21^{\circ}.59\ldots$.

Let $n_{1h}$ be the number of angles determined by two consecutive points of different weight. Each of these angles is at least $\Phi_p\left(1.259,1+p\right)=27^{\circ}.03\ldots$.

Hence, the sum of the $14$ angles around the origin is at least $\mathcal{A}=n_{11}\cdot 31^{\circ}.25+n_{hh}\cdot 21^{\circ}.59+n_{1h}\cdot 27^{\circ}.03$.

We have six possibilities for $[n_{11},n_{hh},n_{1h}]$: $[0, 2, 12], [1, 3, 10], [2, 4, 8], [3, 5, 6], [4,6,4]$, or $[5,7,2]$.

The minimum of $\mathcal{A}$ is reached when $[n_{11},n_{hh},n_{1h}]=[5, 7, 2]$, which corresponds to the case when all points of weight $1$ are consecutive. But even in this case we have $\mathcal{A}=5\cdot 31^{\circ}.25 +7\cdot 21^{\circ}.53+2\cdot 27^{\circ}.03=361^{\circ}.09>360^{\circ}$, a contradiction.
\end{proof}

\begin{fact}\label{f618bkelvin}
It is impossible to have $11$ points of weight $1/2$ in $1\le \rho\le 1.259$.
\end{fact}
\begin{proof}
Suppose that this is possible. The angle between any two such points is at least $\Phi_q\left(1,1.259\right)=32^{\circ}.74\ldots$. However, $11\cdot \Phi_q\left(1,1.259\right)=360^{\circ}.16>360^{\circ}$, a contradiction.
\end{proof}

We are now in position to prove the main result of this section.
\begin{lemma}\label{618kelvin}
It is impossible to have $6$ points of weight $1$ in $p\le \rho\le 1+p$ and $18$ points of weight $1/2$ in $1\le \rho \le 1+p$.
\end{lemma}
\begin{proof}
Assume it is possible. By fact \ref{f618akelvin}, at least $11$ points of weight $1/2$ are in $1\le\rho\le 1.259$. However, this contradicts Fact \ref{f618bkelvin}. Therefore, the proof is complete.
\end{proof}

\section{\bf 7 points of weight 1 and 16 points of weight 1/2}

This case is very similar to the previous one. Again, we need three preliminary facts.
\begin{fact}\label{f716a}
It is impossible to have $7$ points of weight $1$ in $p\le \rho \le 1+p$ and $8$ points of weight $1/2$ in $1.2\le \rho\le 1+p$.
\end{fact}
\begin{proof}
As in the proof of \ref{f618akelvin}, everything reduces to estimate a sum of the form\\
$\mathcal{A}=n_{11}\cdot \Phi_p\left(p,1+p\right)+n_{hh}\cdot \Phi_q\left(1.2,1+q\right)+n_{1h}\cdot \Phi_p\left(1.2,1+p\right)$,
where $\Phi_p\left(p,1+p\right)=31^{\circ}.25\ldots, \\
\Phi_q\left(1.2,1+q\right)= 19^{\circ}.85\ldots, \Phi_p\left(1.2,1+p\right)=24^{\circ}.57\ldots $, and $[n_{11},n_{hh},n_{1h}]$
equals one of the \\following: $[0, 1, 14], [1, 2, 12], [2, 3, 10], [3, 4, 8], [4,5,6], [5,6,4]$, or $[6,7,2]$.

It turns out the minimum of $\mathcal{A}$ is attained when $[n_{11},n_{hh},n_{1h}]=[0, 1, 14]$; that is, the points of weight $1$ and the points of weight $1/2$ alternate in the circular ordering. In this case we have
$\mathcal{A}> 19^{\circ}.85+14\cdot 24^{\circ}.57=363^{\circ}.83$, a contradiction.
\end{proof}
\begin{fact}\label{f716b}
It is impossible to have $11$ points of weight $1/2$ in $1\le \rho\le 1.33$ such that $9$ of them lie in $1\le \rho\le 1.2$.
\end{fact}
\begin{proof}
Suppose that this is possible. Color $9$ of the points in $1\le \rho\le 1.2$ red, and the remaining two points blue. If the two blue
points are not consecutive, then the sum of the $11$ angles is at least
$4\Phi_q\left(1,1.33\right)+7\Phi_q\left(1,1.2\right)=364^{\circ}.606\ldots$. Otherwise, the sum of the angles is at least $3\Phi_q\left(1,1.33\right)+8\Phi_q\left(1,1.2\right)=368^{\circ}.057\ldots$. In either case, we obtain a contradiction.
\end{proof}
\begin{fact}\label{f716c}
It is impossible to have $7$ points of weight $1$ in $p\le \rho \le 1+p$ and $6$ points of weight $1/2$ in $1.33\le \rho\le 1+p$.
\end{fact}
\begin{proof}
As above, everything reduces to estimate a sum of the form\\
$\mathcal{A}=n_{11}\cdot \Phi_p\left(p,1+p\right)+n_{hh}\cdot \Phi_q\left(1.33,1+q\right)+n_{1h}\cdot \Phi_p\left(1.33,1+p\right)$, where \\
$\Phi_p\left(p,1+p\right)=31^{\circ}.25\ldots, \Phi_q\left(1.33,1+q\right)= 22^{\circ}.93\ldots, \Phi_p\left(1.33,1+p\right)=29^{\circ}.32\ldots $, and \\$[n_{11},n_{hh},n_{1h}]$ equals one of the following: $[1, 0, 12], [2, 1, 10], [3, 2, 8], [4, 3, 6], [5,4,4]$, or $[6,5,2]$.

It turns out the minimum of $\mathcal{A}$ is attained when $[n_{11},n_{hh},n_{1h}]=[6, 5, 2]$; that is, the points of weight $1$ are consecutive, and the points of weight $1/2$ are also consecutive in the circular ordering. In this case we have
$\mathcal{A}> 6\cdot31^{\circ}.25+5\cdot 22^{\circ}.93+2\cdot 29^{\circ}.32= 360^{\circ}.79$, a contradiction.
\end{proof}
We are now in position to prove the main result of this section.
\begin{lemma}\label{716}
It is impossible to have $7$ points of weight $1$ in $p\le \rho\le 1+p$ and $16$ points of weight $1/2$ in $1\le \rho \le 1+p$.
\end{lemma}
\begin{proof}
Assume it is possible. By fact \ref{f716a} at least $9$ points of weight $1/2$ are in $1\le\rho\le 1.2$. Fact \ref{f716b} implies that it is not possible to have $11$ points of weight $1/2$ in $1\le \rho \le 1.33$; that is, at least $6$ points of weight $1/2$ are in $1.33\le \rho \le 1+p$. However, fact \ref{f716c} shows that this is impossible, so the proof is done.
\end{proof}

\section{{\bf The harder cases}}

In all the cases covered so far our approach was to restrict the range of the points of weight $1/2$ in order to obtain the desired contradiction. Note that other than using the number of points of weight $1$, we never needed more precise information about their positions: it was sufficient to know that all these points lie in the annulus $p\le \rho \le 1+p$.

The situation is going to be different for the remaining cases. This is to be expected: as the number of points of weight $1/2$ becomes smaller, our reasoning has to take into account the increasing number of points of weight $1$. This is naturally going to lead to slightly more laborious proofs. Fortunately, the intermediate results are easy to check.
\section{\bf 8 points of weight 1 and 14 points of weight 1/2}
\begin{fact}\label{f814a}
It is impossible to have $8$ points of weight $1$ in $p\le \rho \le 1+p$ and $6$ points of weight $1/2$ in $1.21\le \rho\le 1+p$.
\end{fact}
\begin{proof}

As in the proof of \ref{f716a} we estimate a sum of the form
\begin{equation*}
\mathcal{A}=n_{11}\cdot \Phi_p\left(p,1+p\right)+n_{hh}\cdot \Phi_q\left(1.21,1+q\right)+n_{1h}\cdot \Phi_p\left(1.21,1+p\right),
\end{equation*}
where $\Phi_p\left(p,1+p\right)=31^{\circ}.25\ldots, \Phi_q\left(1.21,1+q\right)= 20^{\circ}.17\ldots, \Phi_p\left(1.21,1+p\right)=25^{\circ}.03\ldots $, and $[n_{11},n_{hh},n_{1h}]$ equals one of the following: $[2, 0, 12], [3, 1, 10], [4, 2, 8], [5,3,6], [6,4,4]$, or $[7,5,2]$.

The minimum of $\mathcal{A}$ is attained when $[n_{11},n_{hh},n_{1h}]=[2, 0, 12]$; that is, the points of weight $1/2$ are separated by the points of weight $1$. In this case we have
$\mathcal{A}> 2\cdot 31^{\circ}.25+12\cdot 25^{\circ}.03=362^{\circ}.86$, a contradiction.
\end{proof}

\begin{fact}\label{f814b}
It is impossible to have $1$ point of weight $1$ in the annulus $p \le \rho \le 1.88$  and $9$ points of weight $1/2$ in $1\le \rho\le 1.21$.
\end{fact}
\begin{proof}
Note that this is the first time we are interested in the range of some point of weight $1$. Suppose it is possible to have the points placed as stated. Then the angle between two consecutive points of weight $1/2$ is at least $\Phi_q(1,1.21)=34^\circ.10\ldots$, while the angle
determined by the point of weight $1$ and a point of weight $1/2$ is at least $\Phi_p(1,1.88)=44^\circ.01\ldots$. It follows that the sum of the $10$ angles is at least $8\cdot 34^\circ.10+2\cdot44^\circ.01=360^{\circ}.82$, a contradiction.
\end{proof}

 \begin{fact}\label{f814c}
It is impossible to have $8$ points of weight $1$ in the annulus $1.88 \le \rho \le 1+p$  and $4$ points of weight $1/2$ in $1.29\le \rho\le 1+p$.
\end{fact}
\begin{proof}
With the same conventions as in the previous proofs, the proof reduces to estimating
\begin{equation*}
\mathcal{A}=n_{11}\cdot \Phi_p\left(p,1.88\right)+n_{hh}\cdot \Phi_q\left(1.29,1+q\right)+n_{1h}\cdot \Phi_p\left(1.29,1+p\right),
\end{equation*}
where $\Phi_p\left(p,1.88\right)=34^{\circ}.008\ldots, \Phi_q\left(1.29,1+q\right)= 22^{\circ}.21\ldots, \Phi_p\left(1.29,1+p\right)=28^{\circ}.11\ldots $, and $[n_{11},n_{hh},n_{1h}]$ equals one of the following: $[4, 0, 8], [5,1,6], [6,2,4]$, or $[7,3,2]$.

The minimum of $\mathcal{A}$ is attained when $[n_{11},n_{hh},n_{1h}]=[4, 0, 8]$; that is, the points of weight $1$ are consecutive, and so are the points of weight $1/2$. In this case we have,
$\mathcal{A}> 4\cdot 34^{\circ}+8\cdot 28^{\circ}.11=360^{\circ}.88$, a contradiction.
\end{proof}

\begin{fact}\label{f814d}
It is impossible to have $11$ points of weight $1/2$ in $1\le \rho\le 1.29$ such that $9$ of them lie in $1\le \rho\le 1.21$.
\end{fact}
\begin{proof}
The proof is almost identical to that of fact \ref{f716b}. Color $9$ of the points in $1\le \rho\le 1.21$ red, and the remaining two points blue. If the two blue
points are not consecutive, then the sum of the $11$ angles is at least
$4\Phi_q\left(1,1.29\right)+7\Phi_q\left(1,1.21\right)=366^{\circ}.49\ldots$. Otherwise, the sum of the angles is at least $3\Phi_q\left(1,1.29\right)+8\Phi_q\left(1,1.21\right)=368^{\circ}.66\ldots$. In either case, we obtain a contradiction.
\end{proof}
\begin{lemma}\label{814}
It is impossible to have $8$ points of weight $1$ in $p\le \rho\le 1+p$ and $14$ points of weight $1/2$ in $1\le \rho \le 1+p$.
\end{lemma}
\begin{proof}
Assume it is possible. By fact \ref{f814a} at least $9$ points of weight $1/2$ are in $1\le\rho\le 1.21$. Fact \ref{f814b} implies that all points of weight $1$ are in $1.88\le \rho \le 1+p$. Fact \ref{f814c} then tells us that at least $11$ points of weight $1/2$ lie in $1\le \rho \le 1.29$. However, fact \ref{f814d} shows that it is impossible to have $11$ points of weight $1/2$ in $1\le \rho\le 1.29$ such that $9$ of them lie in $1\le \rho\le 1.21$.
\end{proof}

\section{\bf 9 points of weight 1 and 12 points of weight 1/2}

\begin{fact}\label{f912a}
It is impossible to have $9$ points of weight $1$ in $p\le \rho \le 1+p$ and $4$ points of weight $1/2$ in $1.22\le \rho\le 1+p$.
\end{fact}
\begin{proof}

As in the proof of \ref{f814a} we estimate a sum of the form
\begin{equation*}
\mathcal{A}=n_{11}\cdot \Phi_p\left(p,1+p\right)+n_{hh}\cdot \Phi_q\left(1.22,1+q\right)+n_{1h}\cdot \Phi_p\left(1.22,1+p\right),
\end{equation*}
where $\Phi_p\left(p,1+p\right)=31^{\circ}.25\ldots, \Phi_q\left(1.22,1+q\right)= 20^{\circ}.48\ldots, \Phi_p\left(1.22,1+p\right)=25^{\circ}.47\ldots $, and $[n_{11},n_{hh},n_{1h}]$ equals one of the following: $[5, 0, 8], [6, 1, 6], [7, 2, 4],$ or $[8,3,2]$.

The minimum of $\mathcal{A}$ is attained when $[n_{11},n_{hh},n_{1h}]=[5, 0, 8]$; that is, the points of weight $1/2$ are separated by the points of weight $1$. In this case we have,
$\mathcal{A}> 5\cdot 31^{\circ}.25+8\cdot 25^{\circ}.47=362^{\circ}.01$, a contradiction.
\end{proof}

\begin{fact}\label{f912b}
It is impossible to have $1$ point of weight $1$ in the annulus $p \le \rho \le 1.85$ and $9$ points of weight $1/2$ in $1\le \rho\le 1.22$.
\end{fact}
\begin{proof}
Same proof as for fact \ref{f814b}. The angle between two consecutive points of weight $1/2$ is at least $\Phi_q(1,1.22)=33^\circ.82\ldots$, while the angle determined by the point of weight $1$ and a point of weight $1/2$ is at least $\Phi_p(1,1.85)=44^\circ.76\ldots$. It follows that the sum of the $10$ angles is at least $8\cdot 33^\circ.82+2\cdot44^\circ.76=360^{\circ}.08$, a contradiction.
\end{proof}

\begin{fact}\label{f912c}
It is impossible to have $9$ points of weight $1$ in the annulus $1.85 \le \rho \le 1+p$ and $2$ points of weight $1/2$ in $1.45\le \rho\le 1+p$.
\end{fact}
\begin{proof}
With the same conventions as in the previous proofs, the proof reduces to estimating
\begin{equation*}
\mathcal{A}=n_{11}\cdot \Phi_p\left(p,1.85\right)+n_{hh}\cdot \Phi_q\left(1.45,1+q\right)+n_{1h}\cdot \Phi_p\left(1.45,1+p\right),
\end{equation*}
where $\Phi_p\left(p,1.85\right)=34^{\circ}.008\ldots, \Phi_q\left(1.45,1+q\right)= 23^{\circ}.95\ldots, \Phi_p\left(1.45,1+p\right)=32^{\circ}.06\ldots $, and $[n_{11},n_{hh},n_{1h}]$ equals either  $[7, 0, 4]$, or $[8,1,2]$.

The minimum of $\mathcal{A}$ is attained when $[n_{11},n_{hh},n_{1h}]=[8, 1, 2]$; that is, the points of weight $1/2$ are consecutive.
In this case we have
$\mathcal{A}> 8\cdot 34^{\circ}+23^{\circ}.95+ 2\cdot 32^{\circ}.06=360^{\circ}.07$, a contradiction.
\end{proof}

\begin{fact}\label{f912d}
It is impossible to have $9$ points of weight $1$ in the annulus $1.85 \le \rho \le 1+p$ and $3$ points of weight $1/2$ in $1.24\le \rho\le 1+p$.
\end{fact}
\begin{proof}
Same approach as for the previous fact. The sum of the $12$ angles is bounded from below by a quantity of the form
\begin{equation*}
\mathcal{A}=n_{11}\cdot \Phi_p\left(p,1.85\right)+n_{hh}\cdot \Phi_q\left(1.24,1+q\right)+n_{1h}\cdot \Phi_p\left(1.24,1+p\right),
\end{equation*}
where $\Phi_p\left(p,1.85\right)=34^{\circ}.008\ldots, \Phi_q\left(1.24,1+q\right)= 21^{\circ}.05\ldots, \Phi_p\left(1.24,1+p\right)=26^{\circ}.30\ldots $, and $[n_{11},n_{hh},n_{1h}]$ takes one of the following values:  $[6, 0, 6], [7,1,4]$, or $[8,2,2]$.

The minimum of $\mathcal{A}$ is attained when $[n_{11},n_{hh},n_{1h}]=[6, 0, 6]$; that is, the points of weight $1/2$ are separated by points of weight $1$.
In this case we have
$\mathcal{A}> 6\cdot 34^{\circ}+ 6\cdot 26^{\circ}.3=361^{\circ}.8$, a contradiction.
\end{proof}

\begin{fact}\label{f912e}
It is impossible to have $9$ points of weight $1$ in the annulus $1.85 \le \rho \le 1+p$ and $4$ points of weight $1/2$ in $1.19\le \rho\le 1+p$.
\end{fact}
\begin{proof}
Same approach as for the previous two results. The sum of the $13$ angles is bounded from below by a quantity of the form
\begin{equation*}
\mathcal{A}=n_{11}\cdot \Phi_p\left(p,1.85\right)+n_{hh}\cdot \Phi_q\left(1.19,1+q\right)+n_{1h}\cdot \Phi_p\left(1.19,1+p\right),
\end{equation*}
where $\Phi_p\left(p,1.85\right)=34^{\circ}.008\ldots, \Phi_q\left(1.19,1+q\right)= 19^{\circ}.50\ldots, \Phi_p\left(1.19,1+p\right)=24^{\circ}.08\ldots $, and $[n_{11},n_{hh},n_{1h}]$ takes one of the following values:  $[5, 0, 8], [6,1,6], [7,2,4]$, or $[8,3,2]$.

The minimum of $\mathcal{A}$ is attained when $[n_{11},n_{hh},n_{1h}]=[5, 0, 8]$; that is, the points of weight $1/2$ are separated by points of weight $1$. In this case we have
$\mathcal{A}> 5\cdot 34^{\circ}+ 8\cdot 24^{\circ}=362^{\circ}$, a contradiction.
\end{proof}

\begin{fact}\label{f912f}
It is impossible to have $11$ points of weight $1/2$ in $1\le \rho\le 1.45$ such that $9$ of them lie in $1\le \rho\le 1.19$ and $10$ of them lie in $1\le \rho\le 1.24$.
\end{fact}
\begin{proof}
Color $9$ of the points in $1\le \rho\le 1.19$ red, and the remaining two points blue. If the two blue
points are not consecutive, then the sum of the $11$ angles is at least
$7\Phi_q\left(1,1.19\right)+2\Phi_q\left(1,1.24\right)+2\Phi_q\left(1,1.45\right)> 7\cdot 34^{\circ}.6+2\cdot 33^{\circ}.2+2\cdot 26^{\circ}.3 =361^{\circ}.2$. Otherwise, the sum of the angles is at least $8\Phi_q\left(1,1.19\right)+\Phi_q\left(1,1.24\right)+2\Phi_q\left(1,1.45\right)> 8\cdot 34^{\circ}.6+ 33^{\circ}.2+2\cdot 26^{\circ}.3 =362^{\circ}.6$. In either case, we obtain a contradiction.
\end{proof}
\begin{lemma}\label{912}
It is impossible to have $9$ points of weight $1$ in $p\le \rho\le 1+p$ and $12$ points of weight $1/2$ in $1\le \rho \le 1+p$.
\end{lemma}
\begin{proof}
Assume it is possible. By fact \ref{f912a}, at least $9$ points of weight $1/2$ are in $1\le\rho\le 1.22$. Fact \ref{f912b} implies that all points of weight $1$ are in $1.85\le \rho \le 1+p$. Fact \ref{f912c} then tells us that at least $11$ points of weight $1/2$ lie in $1\le \rho \le 1.45$, and facts \ref{f912d} and \ref{f912e} guarantee that at least $10$ of these lie in $1\le \rho \le 1.24$, and at least $9$ lie in $1\le \rho \le 1.19$. However, fact \ref{f912f} shows that this latest situation is impossible.
\end{proof}


\section{\bf 10 points of weight 1 and 10 points of weight 1/2}

This is by far the most delicate case. The estimates must be much more precise in order to produce the desired contradiction.
This is where we will finally be able to explain why the choice $p=1.409$. Fortunately, the approach is almost identical to the one in the previous section.

\begin{fact}\label{f1010a}
It is impossible to have $10$ points of weight $1$ in $p\le \rho \le 1+p$ and $2$ points of weight $1/2$ in $1.2931\le \rho\le 1+p$.
\end{fact}
\begin{proof}
We need to estimate the sum
\begin{equation*}
\mathcal{A}=n_{11}\cdot \Phi_p\left(p,1+p\right)+n_{hh}\cdot \Phi_q\left(1.2931,1+q\right)+n_{1h}\cdot \Phi_p\left(1.2931,1+p\right),\quad \text{where}
\end{equation*}
$\Phi_p\left(p,1+p\right)=31^{\circ}.2555\ldots, \Phi_q\left(1.2931,1+q\right)= 22^{\circ}.2806\ldots, \Phi_p\left(1.2931,1+p\right)=28^{\circ}.2107\ldots $, and $[n_{11},n_{hh},n_{1h}]$ equals either $[8, 0, 4]$ or $[9,1,2]$.
The minimum of $\mathcal{A}=360^{\circ}.002\ldots$ is reached when $[n_{11},n_{hh},n_{1h}]=[9, 1, 2]$,
$\mathcal{A}> 9\cdot 31^{\circ}.2555+22^{\circ}.2806+2\cdot 28^{\circ}.2107=360^{\circ}.0015$, a contradiction.
\end{proof}

\begin{fact}\label{f1010b}
It is impossible to have $1$ point of weight $1$ in the annulus $p \le \rho \le 1.59$ and $9$ points of weight $1/2$ in $1\le \rho\le 1.2931$.
\end{fact}
\begin{proof}
Same reasoning as for fact \ref{f912b}. The angle between two consecutive points of weight $1/2$ is at least $\Phi_q(1,1.2931)=31^\circ.8557\ldots$, while the angle determined by the point of weight $1$ and a point of weight $1/2$ is at least $\Phi_p(1,1.59)=52^\circ.6013\ldots$. It follows that the sum of the $10$ angles is at least $8\cdot 31^\circ.8557+2\cdot52^\circ.6013=360^{\circ}.0482$, a contradiction.
\end{proof}

The above two results imply that if we are to have $10$ points of weight $1$ and $10$ of weight $1/2$, then all points of weight $1$ are in the annulus $1.59\le \rho \le 1+p$. The following statements are just analogs of facts \ref{f912c}, \ref{f912d}, and \ref{f912e} from the previous section. In order to save space, we group them together.

\begin{fact}\label{bigfact}
It is impossible to have $10$ points of weight $1$ in the annulus $1.59 \le \rho \le 1+p$ and any of the following
\begin{itemize}
\item[a.]{$1$ of weight $1/2$ in $1.2571\le \rho \le 1+p$.}
\item[b.]{$2$ of weight $1/2$ in $1.1513\le \rho\le 1+p$.}
\item[c.]{$3$ of weight $1/2$ in $1.1254\le \rho \le 1+p$.}
\item[d.]{$4$ of weight $1/2$ in $1.1138\le \rho\le 1+p$.}
\item[e.]{$5$ of weight $1/2$ in $1.1072\le \rho \le 1+p$.}
\end{itemize}
\end{fact}
\begin{proof}
Part a. follows from $9\cdot \Phi_{p}(1.59, 1+p)+2\cdot \Phi_p(1.2571,1+p)=360^{\circ}.0067\ldots.$
For part b. the sum of the angles is at least $n_{11}\Phi_p(1.59, 1+p)+n_{hh}\Phi_q(1.153,1+q)+n_{1h}\Phi_p(1.1513,1+p)$ where
$[n_{11},n_{hh},n_{1h}]\in\{[8,0,4], [9,1,2]\}$. The minimum is $360^{\circ}.022\ldots$ , when $[n_{11},n_{hh},n_{1h}]= [8,0,4]$.
Similarly, in case c. the sum of the angles is at least $n_{11}\Phi_p(1.59, 1+p)+n_{hh}\Phi_q(1.1254,1+q)+n_{1h}\Phi_p(1.1254,1+p)$ where
$[n_{11},n_{hh},n_{1h}]\in\{[7,0,6], [8,1,4], [9,2,2]\}$. The minimum is $360^{\circ}.021\ldots$ , when $[n_{11},n_{hh},n_{1h}]= [7,0,6]$.
Likewise, in case d. the sum of the angles is at least $n_{11}\Phi_p(1.59, 1+p)+n_{hh}\Phi_q(1.1138,1+q)+n_{1h}\Phi_p(1.1138,1+p)$ where
$[n_{11},n_{hh},n_{1h}]\in\{[6,0,8],[7,1,6], [8,2,4], [9,3,2]\}$. The minimum is $360^{\circ}.036\ldots$ , when $[n_{11},n_{hh},n_{1h}]= [6,0,8]$. Finally, in part e. the sum is at least $n_{11}\Phi_p(1.59, 1+p)+n_{hh}\Phi_q(1.1072,1+q)+n_{1h}\Phi_p(1.1072,1+p)$ where $[n_{11},n_{hh},n_{1h}]\in\{[5,0,10], [6,1,8],[7,2,6], [8,3,4], [9,4,2]\}$. The minimum is $360^{\circ}.033\ldots$ , when $[n_{11},n_{hh},n_{1h}]= [5,0,10]$.
\end{proof}

Finally, we can prove the hardest case
\begin{lemma}\label{1010}
It is impossible to have $10$ points of weight $1$ in $p\le \rho \le 1+p$ and $10$ points of weight $1/2$ in $1\le \rho \le 1+p$.
\end{lemma}
\begin{proof}
Assume it is possible. Then from fact \ref{bigfact} it follows that all $10$ points of weight $1/2$ are in $1\le \rho \le 1.12571$, at least $9$ of which are in $1\le \rho \le 1.1513$, at least $8$ of which are in $1\le \rho \le 1.1254$, at least $7$ of which must be in $1\le \rho \le 1.1138$, and at least $6$ of which must lie in $1\le \rho \le 1.1072$. Color $6$ of the points in $1\le \rho \le 1.1072$ red, and the remaining $4$ points blue. Analyzing the $10!/6!=5040$ different permutations shows that the smallest possible sum is obtained when no two blue points are consecutive. In this case we have that the sum of the angles is at least
\begin{equation*}
2\Phi_q(1,1.072)+2\Phi_q(1,1.1138)+2\Phi_q(1,1.1254)+2\Phi_q(1,1153)+2\Phi_q(1,1.2571)=360^{\circ}.0047\ldots.
\end{equation*}
This produces the contradiction.
\end{proof}

\begin{obs} It should now be clear why this case was so difficult. In every single step, the lower bounds for the sum of the angles are barely over $360^{\circ}$. This is also the reason that the choice $p=1.409$ is crucial. If we change $p$ to $1.408$ or to $1.410$, the above argument fails.
\end{obs}
This leaves us with the last case, which fortunately is not that difficult.

\section{\bf 11 points of weight 1 and 8 points of weight 1/2}
\begin{fact}\label{f118a}
It is impossible to have $11$ points of weight $1$ in $p\le \rho \le 1+p$ and any of the following
\begin{itemize}
\item[a.]{$1$ of weight $1/2$ in $1.2\le \rho \le 1+p$.}
\item[b.]{$2$ of weight $1/2$ in $1.12\le \rho\le 1+p$.}
\end{itemize}
\end{fact}
\begin{proof}
Part a. follows easily from $10\cdot \Phi_p\left(p,1+p\right)+2\cdot \Phi_p\left(1.2,1+p\right)>10\cdot 31^{\circ}.25+2\cdot 24^{\circ}.57=361^{\circ}.64$, a contradiction.

For part b. we need to estimate the sum
\begin{equation*}
\mathcal{A}=n_{11}\cdot \Phi_p\left(p,1+p\right)+n_{hh}\cdot \Phi_q\left(1.12,1+q\right)+n_{1h}\cdot \Phi_p\left(1.12,1+p\right),
\end{equation*}
where $\Phi_p\left(p,1+p\right)=31^{\circ}.25\ldots, \Phi_q\left(1.12,1+q\right)= 16^{\circ}.40\ldots, \Phi_p\left(1.12,1+p\right)=19^{\circ}.94\ldots $, and $[n_{11},n_{hh},n_{1h}]$ equals either $[9, 0, 4]$, or $[10,1,2]$.
The minimum of $\mathcal{A}=361^{\circ}.09\ldots$ is reached when $[n_{11},n_{hh},n_{1h}]=[9, 0, 4]$, a contradiction.
\end{proof}

\begin{fact}\label{f118b}
It is impossible to have $11$ points of weight $1$ all the annulus $1.5 \le \rho \le 1+p$.
\end{fact}
\begin{proof}
This is immediate since $11\cdot \Phi_p(1.5,1+p) =361^{\circ}.88\ldots$, a contradiction.
\end{proof}

\begin{fact}\label{f118c}
It is impossible to have $1$ point of weight $1$ in the annulus $p \le \rho \le 1.5$ and $8$ of weight $1/2$ in $1\le \rho \le 1.2$ such that
$7$ of them lie in the annulus $1\le \rho\le 1.12$.
\end{fact}
\begin{proof}
Color $7$ points of weight $1/2$ that lie in $1\le \rho \le 1.12$ red, and the remaining two points blue.
If the blue points are not consecutive, then the angle sum is at least $5\Phi_q(1,1.12)+2\Phi_p(1,1.5)+2\Phi_q(1,1.2)=365^{\circ}.72\ldots$.
Otherwise, the angle sum is at least $6\Phi_q(1,1.12)+2\Phi_p(1,1.5)+\Phi_q(1,1.2)=368^{\circ}.11\ldots$. In either case, we obtain a contradiction.
\end{proof}

\begin{lemma}\label{118}
It is impossible to have $11$ points of weight $1$ in $p\le \rho \le 1+p$ and $8$ points of weight $1/2$ in $1\le \rho \le 1+p$.
\end{lemma}
\begin{proof}
Assume it is possible. Fact \ref{f118a} implies that all $8$ points of weight $1/2$ are in $1\le \rho \le 1.2$, and at least $7$ of them are in $1\le \rho \le 1.12$. On the other hand, fact \ref{f118b} guarantees the existence of a point of weight $1$ in the annulus $p\le \rho \le 1.5$.  However, fact \ref{f118c} shows that such combination is impossible, so the proof is complete.
\end{proof}

We have finally reached the end of the road. The only thing left to do is to combine the Lemmata \ref{120}, \ref{027}, \ref{224}, \ref{421}, \ref{520}, \ref{618kelvin}, \ref{716}, \ref{814}, \ref{912}, \ref{1010} and \ref{118} to complete the proof of Theorem \ref{weight}.

\section{\bf Conclusions and directions for future research}
In this paper we presented a small improvement on the upper bound of the size of the Euclidean sphere of influence graph.
We feel that we reached the limits of the present method. For further progress, new ideas seem to be needed.
One may consider defining a more refined weight function, with points of several weights rather than just two.
This may very well work, but there is a price to be paid: the resulting case analysis is going to be even more tedious.

Many interesting questions remain unanswered. The {\it thickness} of a graph $G$ is the minimum number of planar graphs whose union is $G$. Allgeier and Kubicki \cite{AK} conjectured that the thickness of SIGs is bounded. The complete graph $K_8$ is an OSIG and has thickness $2$, but $K_9$ has thickness $3$. At the present time, no SIG is known with thickness more than $2$.

According to Toussaint \cite{toussaint14}, sphere of influence graphs have applications in low-level computer vision, cluster analysis, pattern recognition, geographic information systems, and even marketing. He provides an extensive list of references.

This is why we believe that the study of sphere of influence graphs is both interesting and potentially useful.

\thispagestyle{empty}
{\footnotesize{
}}
\end{document}